\documentclass{article}
\usepackage{float}
\usepackage{comment}
\usepackage{multirow}
\usepackage{epsfig}
\usepackage{dcolumn}
\usepackage{bm}

\newcommand{\be}{\begin{eqnarray}}
\newcommand{\ee}{\end{eqnarray}}

\usepackage[english]{babel}
\usepackage{amssymb}
\usepackage{amsmath}
\usepackage{amsfonts}
\usepackage{amsbsy}
\usepackage{indentfirst}
\usepackage{graphicx}
\usepackage{color}
\usepackage{xcolor}
\usepackage{cite} 
\usepackage{mathtools}
\newcommand{\C}{\mathbb{C}}
\newcommand{\R}{\mathbb{R}}

\newcommand{\e}{\mathrm{e}}
\def\diag{{\rm diag}}
\def\Tr{{\rm Tr}}

\DeclarePairedDelimiter{\ceil}{\lceil}{\rceil}
\DeclarePairedDelimiter{\floor}{\lfloor}{\rfloor}

\newcommand{\beq}{\begin{equation}}
\newcommand{\eeq}{\end{equation}}

\newcommand{\bmat}{\begin{displaymath}}
\newcommand{\emat}{\end{displaymath}}

\def\1{{\bf 1}}

\newtheorem{theorem}{Theorem}[section]
\newtheorem{pro}{Proposition}[section]

\newtheorem{corol}{Corollary}[section]

\newtheorem{Exa}{Example}[section]
 
\newenvironment{proof}[1][Proof]{\noindent\textbf{#1.} }{\ \rule{0.5em}{0.5em}}

\begin{document}

\title{ Krein Space Numerical Range of Block Matrices - a Unified Approach to the Hyperbolic Case}




\begin{center}
{\Large Krein Space Numerical Range of Block Matrices - a Unified Approach to the Hyperbolic Case}

\smallskip

N. Bebiano\footnote{CMUC and  Mathematics Department,
 University of Coimbra,
3001-501 Coimbra (bebiano@mat.uc.pt)}, R. Lemos\footnote{CIDMA, Department of Mathematics, University of
Aveiro, Aveiro, Portugal (rute@ua.pt)} and G.
Soares\footnote{CMAT-UTAD, Department of  Mathematics, University of
Tr\'as-os-Montes  e Alto Douro,  
 Vila Real, Portugal (gsoares@utad.pt)}
 
 \end{center}
\begin{abstract} 
In this paper,  we  investigate the Krein space numerical range of $2$-by-$2$ block matrices, with diagonal blocks as scalar multiples of the identity. For these matrices, we specifically investigate the cases when the respective boundary generating curves consist of hyperbolas. This provides a unified 
approach to derive established and new results concerning the numerical range hyperbolic shape.
\end{abstract}

{\bf subjclass} {Primary 47A12; Secondary 15A60}

{\bf Keywords:} {Numerical range, Krein space, Indefinite Kippenhahn curve, Hyperbolical Range Theorem,  Block matrix}


\section{Introduction}
 
 As usual, let $M_n$ stand for the associative algebra of
$n\times n$ complex matrices and let $I_n$ denote the identity matrix.
Consider the complex vector space ${\bf C}^n$ equipped with the Krein space
structure  endowed by the {\it indefinite inner product}  
$[x,y]_J = y^*J\,x$, for $x, y \in {\bf C}^n$, with $J\!= I_r \oplus -I_{n-r}$, $0\leq r\leq n$ (for references on Krein spaces see e.g.\! \cite{azizov, LR}).  In general, we may consider 
a non-singular Hermitian  matrix $H \in M_n$ instead of $J$. 
 The {\it Krein space numerical range},  also called the {\it indefinite numeri\-cal range},  
of $A\in M_n$ is  defined as
$$W^H(A)\,=\,-W^{H}_-(A)\,\cup\, W^H_{\,+}(A)$$
with
$$
  W_{\pm }^H(A)\,=\,\big\{[Ax,x]_H:\ x\in\C^n, \ [x,x]_H={\pm 1}\big\}.
$$
It is clear that $W^{-H}_{\,+}(-A)\,=\, W^H_-(A)$. 

 If $H=I_n$, then $W^H(A)=W^H_+(A)$ reduces to the well-known classical {\it numerical
  range},  introduced a century ago   by    Toeplitz \cite{Toep} and Hausdorff   \cite{Haus},  and extensively investigated thereafter (see e.g.  \cite[Chapter 1]{HJTopics}  and references therein). 
The set $W^H(A)$ has also been researched and applied by some authors (see e.g. \cite{BLPS SIAM Proc, BLPS,  BNLPS2005, BNLPS, BPNPS2015, BLS_Hyp, BLSlast, ELA, ricardo0, ricardo, ELA quadr, GLS, LR PAMS, LR ELA3, CTU, PT, BLP2004}). 

 It is well-known that $W(A)$ contains the spectrum of $A$. In the present context for
$$
  \sigma^H_{\pm}(A)\,=\,\big\{\lambda\,{\in\mathbb C}:\ \exists\ x\in\C^n,\ [x,x]_H=\pm 1, \ A\,x=\lambda\, x\big\},
$$ 
the following inclusions hold: $\sigma^H_{+}(A)\subset  W^J_{+}(A)$ and $\sigma^H_{-}(A)\subset - W^H_{-}(A).$
 

Both sets $W_+^H(A), W_-^H(A)$ are convex, nevertheless they may not be closed or bounded~\cite{CTU}. On the other hand, $W^H(A)$  is {\it pseudo-con\-vex}, that is, for any pair of distinct  points $w_1,w_2\in W^H(A)$, either 
$W^H(A)$ contains the closed line segment 
$\{tw_1+(1-t)w_2: ~0\leq t\leq 1\}
$ 
or $W^H(A)$ contains the half-lines 
$
\{t w_1+ (1-t) w_2 : ~t\leq 0 ~ \hbox{or} ~ t\geq 1\}.
$

A matrix $A\in M_n$ is 
$H$-{\it Hermitian} if it equals  its $H$-{\it adjoint}, which is given by $A^{[*]}\!=H^{-1}A^*H,$
and  $W^H(A) \subseteq \mathbb R$ if and only of $A$ is $H$-Hermitian \cite{CTU}.

If $r$ is the number of positive eigenvalues of the matrix $H$, counting multiplicities, then by Sylvester law of inertia 
there exists $S\in M_n$ non-singular, such that $S^*HS= J$, where
 $J=I_r \oplus -I_{n-r}$ is the inertia matrix of $H$, and $W^H(A)=W^J(S^{-1}AS)$. 
We also remark that  $A$ is $H$-Hermitian if and only if $S^{-1}AS$ is $J$-Hermitian.
 Thus, without loss of generality, we may restrict our attention to the study of $W^J(A)$ and, for simplicity, the notation 
 $A^\#=JA^*J$ for the $J$-adjoint of $A$ is used.

Any matrix $A\in M_n$ admits the representation
$
 A=\Re^J(A)+i\,\Im^J(A),
$
where $$\Re^J(A) =\frac{A+A^\#}{2}\qquad \hbox{and}\qquad
\Im^J(A) =\frac{A-A^\#}{2i}$$
are $J$-Hermitian matrices. 
For  each angle $\theta\in \R$, we define the matrix
\begin{equation}\label{Htheta}
 H_\theta(A)\,=\, \Re^H(A) \cos\theta+\Im^H(A)\sin\theta,
\end{equation}
which is $J$-Hermitian, so it has real or complex eigenvalues, these occurring in complex conjugate pairs  \cite{LR}. If  $H_\theta(A)$ has 
non-real eigenvalues, then $W^J(H_\theta(A))$ is the whole real line \cite[Proposition~2.1]{BNLPS2005}. 

To avoid trivial cases of degeneracy of $W^J(A)$, we  will focus on the class ${\mathcal J}$ of matrices $A\in M_n$ such that $H_\theta(A)$ for $\theta$ in some real interval, has real eigen\-va\-lues, satisfying
\begin{equation}\label{sigma+}
\sigma^J_+\big(H_\theta(A)\big)=\{\lambda_1(\theta),\ldots,\lambda_r(\theta)\}, \quad \lambda_1(\theta)\geq\cdots\geq \lambda_r(\theta),
\end{equation} 
\begin{equation}\label{sigma-}
\sigma^J_-\big(H_\theta(A)\big)=\{\lambda_{r+1}(\theta),\ldots,\lambda_n(\theta)\}, \quad \lambda_{r+1}(\theta)\geq\cdots\geq \lambda_n(\theta)
\end{equation} 
and the elements in
$\sigma_+^J(H_\theta(A))$ and $\sigma_-^J(H_\theta(A))$ do {\it not interlace}, that is, 
either 
$\lambda_r(\theta)>\lambda_{r+1}(\theta)$ 
or
$\lambda_n(\theta)>\lambda_1(\theta)$.  
If the eigenvalues of $H_\theta(A)$ interlace, then $W^J(H_\theta(A))$ is also the whole real line. 


The characteristic polynomial  of $H_\theta(A)$, 
 that is, 
$$p_{A}(z, \theta) \,=\, {\rm det}\big(\Re^J\!(A)\cos\theta\,+\,\Im^J\!(A)\sin\theta\,-\,z\,I_n\big),$$
is called the {\it indefinite  generating polynomial} 
of $W^J(A)$ and its eigenvalues play a crucial role in our study (for details, see Section~2).  For $A\in M_n$, through the equation 
${\rm det}(u\,\Re^J\!(A)\,+\,v\,\Im^J\!(A)\,+\,w\,I_n)=0$, it is associated an  algebraic curve  of  class $n$, in homogeneous line
coordinates. Its real part, denoted by $C^J(A)$, is the {\it indefinite  boundary generating curve}, or {\it indefinite Kippenhahn curve}, of $W^J(A)$ and its
pseudo-convex hull gene\-ra\-tes the set $W^J(A)$ as follows. For any two points $w_1=[Ax,x]_J$ and $w_2=[Ay,y]_J$ in $C^J(A)$, we take the line segment joining them, if $[x,x]_J\,[y,y]_J>0$, or the  union of two half-rays given by $\{tw_1+(1-t)w_2: t\leq 0$ or $t\geq 1\}$, if $[x,x]_J\,[y,y]_J<0$ (see e.g.\! \cite{BLPS} for more information). 


In general, not too much is known on  the boundary of $W^J(A)$,  
except for the case $n=2$  and for  certain structured matrices, some of them with low order, when $W^J(A)$ is the pseudo-convex hull of a finite number of hyperbolas \cite{BLP2004,BNLPS2005, BLS_Hyp, BLSlast, ricardo,ricardo0, ELA}. In this paper, we are interested in investigating to which extent this shape persists.

 For $\alpha, \beta \in \mathbb{C}$, we will  focus on  matrices of the form
\begin{equation}\label{bloco}
A=\left[\begin{array}{cc}
\alpha I_{r}  & C\\
D & \beta I_{n-r}
\end{array}\right],\qquad \alpha\neq \beta,
\end{equation}
with at least a non-zero off-diagonal entry.
If $C=D=O$, then $W^J(A)$ is the  union of the  half-lines 
$
\{t\,\alpha+ (1-t) \beta: ~t\leq 0 ~ \hbox{or} ~ t\geq 1\},
$
with $\alpha\in W_+^J(A)$ and $\beta\in -W_-^J(A)$.

The remaining of this note is organized as follows. Section~\ref{Prereq} contains preli\-minary results used in subsequent discussions. In Sections~3, 4 and 5, 
counterparts of  classical numerical range results (e.g. \cite{Chien, OAM2013, Geryba, Li2011, Linden}) 
are presented in the context of indefinite inner product spaces, specifically in the cases when the indefinite boundary generating curve consists of hyperbolas.  
Previous known results concerning the shape of $W^J(A)$, when $C^J(A)$ contains a family of hyperbolas, 
are in particular reobtained from the main unified result in Section~3.  Most results are illustrated via specific numerical examples.
Section 4 is dedi\-cated to the case when arcs of hyperbolas and flat portions appear at the boundary of $W^J(A)$. Section 5
presents families of matrices with hyperbolic $W^J(A)$,  derived from the main result,  except for the last theorem.
Section 6 is applied to the Krein space numerical range of certain tridiagonal matrices.  
In Section~7, some  final comments are stated.

\smallskip

\section{Pre-requisites}\label{Prereq}

\smallskip

The  {\it Hyperbolical Range Theorem}  \cite{BLPS, BLS_Hyp} describes the indefinite numerical range in the $2$-by-$2$ case. It states that if $A\in M_2$ has  eigenvalues $\lambda_1, \lambda_2$ and $J=\diag(1,-1)$,  then 
$W^J(A)$ is bounded by the non-degenerate  hyperbola,
with foci at 
$\lambda_1, \lambda_2$, 
transverse and non-transverse axes of lengths
$$
  \left(\Tr(A^\#A)-2\Re(\lambda_1\bar\lambda_2)\right)^\frac{1}{2} 
  \qquad \hbox{and}\qquad 
  \left(|\lambda_1|^2+|\lambda_2|^2-\Tr(A^\#A)\right)^\frac{1}{2},
$$
respectively, if and only if 
\begin{equation}\label{inv cond hip}
  2\Re(\lambda_1\bar\lambda_2)<\Tr(A^\#A)< |\lambda_1|^2+|\lambda_2|^2. 
\end{equation}
Moreover,
 $W^J(A)$ is the disjoint union of two closed half-lines,  with endpoints at $\lambda_1, \lambda_2$,  on the line containing these points,
 if and only if 
 $$
 2\Re(\lambda_1\bar\lambda_2)<\Tr(A^\#A) = |\lambda_1|^2+|\lambda_2|^2.
 $$
The  reduction to a singleton occurs  if and only if $A$ is a scalar matrix.
Further, $W^J(A)$ may degenerate into  a line, eventually except a point, or
the whole complex plane, eventually  except a line (see \cite{BLPS SIAM Proc} for a detailed description). 


 The following elementary properties  will be used throughout. For
 $A \in M_n$, 
\begin{itemize}
  \item[P1.] $W^J(\alpha A+\beta I_n)\!=\alpha W^J(A)+\beta$ for any $\alpha, \beta\in \C$ ({\it translational property});
  \item[P2.]   $W^J(A)$ is $J$-{\it unitarily invariant}: 
      $W^J(U^\#AU)\!=\!W^J(A)$  for any $J$-uni\-tary $U\in M_n$, that is, $UU^\#=I_n$.
\end{itemize}

In the sequel, let $J=I_{r}\oplus -I_{n-r}$, $0<r<n$.
In order to state a criterion of non-degenerate hyperbolicity of $W^J(A)$ obtained in \cite{BLS_Hyp}, and which plays a key role in the present research,
 let $\Omega$ be the interval,  with greatest possible dia\-meter,  of angles $\theta$
such that $H_\theta(A) \in {\mathcal J}$. For simplicity, recalling \eqref{sigma+} and \eqref{sigma-},  we adopt the following unified notation: 
\begin{itemize} 
\item[(I)] $\lambda_{L} (\theta)=\lambda_{r+1}(\theta)$ and $\lambda_{R}(\theta)=\lambda_{r}(\theta)$ if $\lambda_r(\theta)>\lambda_{r+1}(\theta)$ holds for $\theta \in \Omega$;
\end{itemize}
\begin{itemize} 
\item[(II)] $\lambda_{L} (\theta)=\lambda_{1}(\theta)$ and $\lambda_{R}(\theta)=\lambda_{n}(\theta)$ if $\lambda_n(\theta)>\lambda_{1}(\theta)$ holds for $\theta \in \Omega$.
\end{itemize}
If $W^J(A)$ is a  hyperbolic disc with horizontal transverse axis, then
$x=\lambda_{L} (\theta)$ and $x=\lambda_{R} (\theta)$ are support  lines, respectively, of $-W^J_-(A)$ and $W^J_+(A)$, in case (I),
or of  $W^J_+(A)$ and $-W^J_-(A)$, in case (II).   
The envelopes of the two following family of lines: 
\[
\e^{-i\theta}(\lambda_R(\theta)+i\R)\,, \qquad \e^{-i\theta}(\lambda_L(\theta)+i\R),
\]
for $\theta$ ranging over $\Omega$, provide the characterization  of the boundaries of $W^J_{+}(A)$ and $-W^J_{-}(A)$, and ultimately of the set  $W^J(A)$. 

\medskip

\begin{theorem}\label{T2.11} 
\noindent {\bf (Criterion of hyperbolicity)}
Let $\widetilde{a},\widetilde{b}>0$ and $A\in M_n$. 
The set $W^J(A)$ is bounded by  the non-degenerate hyperbola centered at the origin, with horizontal transverse  
and vertical non-transverse semi-axes 
of length $\widetilde a$ and $\widetilde b$, respectively,  if and only if
\[
 \lambda_{R}\big(\theta\big)=\left(\widetilde a^2 - \widetilde c^2 \sin^2\theta\right)^{\frac{1}{2}}
  \qquad   \hbox{and} \qquad  
 \lambda_{L}\big(\theta\big)=-\left(\widetilde a^2 - \widetilde c^2 \sin^2\theta\right)^{\frac{1}{2}}
\]
where $\widetilde c^2= \widetilde a^2+ \widetilde b^2$,  for  $\theta \in (-\theta_0, \theta_0)$ with $\theta_0=\arctan(\widetilde a/\widetilde b)$. 
\end{theorem}

\medskip




Without loss of generality, for the study of $W^J(A)$ for matrices $A$ of type (\ref{bloco}) 
we can assume throughout $n\leq 2r$.  Indeed, if $n>2r$, we can instead consider  $W^{\!J'}\!(A')$ for the matrices
 \begin{equation}\label{bloco1}
A' = \left[\begin{array}{cc}
  \beta I_{n-r}  & D\\
  C & \alpha I_{r}
\end{array}\right] \qquad \hbox{and}
\qquad J' =I_{n-r}\oplus -I_r;
\end{equation}
becasuse by an adequate permutation matrix $P$, we have
$A' = P^TA\,P$ and $J'= P^T(-J)\, P,$ 
which guarantees the equalities: 
$$W^J(A)\,=\,W^{-J}(A)\,=\,W^{J'}\!(A'),$$
although interchanging the roles of the convex components, according to  $$W_+^J(A)= -W_-^{J'}(A') \qquad \hbox{and} \qquad -W_-^J(A)= W_+^{J'}(A').$$

For matrices  of the form {\rm (\ref{bloco})}, the next result gives pertinent information on the eigenvalues of the $J$-Hermitian matrix $H_\theta (A)$, for each $\theta\in \mathbb R$,
in terms of the eigenvalues  of the positive semidefinite matrix: 
\begin{equation}\label{Mtheta}
   M(\theta)\,=\, C^*C+D\,D^*-2\Re(\e^{-2i\theta}DC),\qquad \hbox{if} \quad n\leq 2r.
\end{equation}

If $n>2r$, we just need to  interchange $\alpha, r, C$  and $\beta, n-r, D$, respectively, having in mind the considerations above.
For simplicity of notation,   let
\begin{equation}\label{B_e_omega}
 B=A-\frac{1}{2}(\alpha+\beta)I_n \qquad \hbox{and}  \qquad \omega =\frac{1}{2}(\alpha-\beta).
\end{equation}


\begin{pro}\label{lema1} Let $A\in M_n$ be a block matrix of type {\rm (\ref{bloco})}, $n \leq 2r$ 
and $\theta\in \mathbb R$. Let  $\mu_1(\theta),\dots,\mu_{n-r}(\theta)$ be the eigenvalues of the matrix $M(\theta)$ in {\rm (\ref{Mtheta})}.
\begin{itemize}
\item[\bf (a)] If $n<2r$,  the eigenvalues of $H_\theta(A)$  are $\Re(\e^{-i\theta}\alpha)$ in $\sigma_+^J(H_{\theta}(A))$ and
\begin{equation}\label{lambdaj}
\lambda_{j,\pm}(H_\theta(A))\,=\,\frac{1}{2}\Re(\e^{-i\theta}(\alpha+\beta))
                      \pm \sqrt{\big(\Re(\e^{-i\theta}\omega)\big)^2-\frac{1}{4}\mu_j(\theta)}
\end{equation} 
for $j=1, \dots, n-r$. 
If $n=2r$, the eigenvalues of $H_\theta(A)$  are those given in {\rm (\ref{lambdaj})} for $j=1, \dots, \frac{n}{2}$.

\item[\bf (b)] If  $\mu_j(\theta)< 4\big(\Re(\e^{-i\theta}\omega)\big)^2$, then
 one of the eigenvalues in {\rm (\ref{lambdaj})} belongs to $\sigma_+^J(H_{\theta}(A))$ and the other to $\sigma_-^J(H_{\theta}(A))$.
 
\end{itemize}
\end{pro}

\begin{proof} \noindent {\bf (a)} 
Since $A$ is the block matrix {\rm (\ref{bloco})}, considering (\ref{B_e_omega}), we have 
$$
B = \left[\begin{matrix}\omega I_{r}  & C\\
D & -\omega I_{n-r}
\end{matrix}\right]
$$
and
\begin{equation}\label{HthetaAB}
 H_\theta(A)=\frac{1}{2}\Re(\e^{-i\theta}(\alpha+\beta))I_n+H_\theta(B).
\end{equation}

Let $N_\theta=\frac{1}{2}\left(\e^{-i\theta}C-\e^{i\theta}\,D^*\right)$ and consider $k={\rm rank}\,N_{\theta}$.
By the singular value decomposition, there exist unitary matrices $U\in M_r$, $V\in M_{n-r}$, such that  
  $$
    U^*\,N_\theta\, V \, = \left[\begin{matrix} D_{\theta} \\ O\end{matrix}\right],
   $$  
 where $D_\theta$  is diagonal with the non-zero singular values   $s_1(\theta), \dots, s_k(\theta)$
of $N_\theta$ on the first main diagonal entries  (and the null block $O$ is absent if $n=2r$). 
 Since
\[  
H_{\theta}(B)\,
=\,\left[\begin{matrix}
\omega_\theta I_{r}  & N_\theta\\
-N_\theta^* & -\omega_\theta I_{n-r}
\end{matrix}\right], \qquad 
\omega_\theta=\Re(\e^{-i\theta}\omega),
\] 
considering  $Z=U\oplus V$, we can see that the matrix 
\begin{equation}\label{Ftheta}
Z^* H_{\theta}(B)\, Z    
  \, =\, \left[\begin{matrix}\,\omega_\theta I_r& U^*N_\theta\,V\\
    -\,V^*N_\theta^*\,U&-\omega_\theta I_{n-r}\,\end{matrix}\right]. 
\end{equation}
is  permutationally similar to a direct sum of $k$  blocks of order $2$
 of type  
\[
   \left[\begin{matrix} \omega_\theta & \ s_i(\theta) \ \\
   -s_i(\theta) & -\omega_\theta \end{matrix}\right], \qquad i=1, \dots, k,
\]
and blocks of order $1$, namely $r-k$ equal to  $\omega_\theta$ and  $n-r-k$  equal to  $-\omega_\theta$.
Hence, the eigenvalues of $H_\theta(B)$ are given by
\begin{equation}\label{vpmu}
  \lambda_{j,\pm}(\theta)=
  \pm\sqrt{\omega^2_\theta - s_j^2(\theta)},\qquad j=1,\ldots,k, 
\end{equation}
together with $\omega_\theta$ and $-\omega_\theta$ of multiplicities $r-k$ and $n-r-k$, respectively.
We remark that  $M(\theta)=4\,N^*_\theta N_\theta$. 
Thus, the eigenvalues of $M(\theta)$ are  $\mu_j(\theta)=4s^2_j(\theta)$, $j=1, \dots, n-r$. Recalling (\ref{HthetaAB}),  the eigenvalues of $H_\theta(A)$ are readily obtained. 
\smallskip

\noindent {\bf (b)} 
Under the hypothesis, we have $s_j^2(\theta)<\omega^2_\theta$, so  (\ref{vpmu}) are real and distinct. 
Suppose that $s_j(\theta)\neq 0$. 
The vector $u_{j,\pm}(\theta)$ whose $j$th entry is $\omega_\theta+\lambda_{j,\pm}(\theta),$ 
the $(r+j)$th entry is $-s_j(\theta)$ and all the others are zero is an  eigenvector of (\ref{Ftheta}) associated to the corresponding eigenvalue (\ref{vpmu}).
By some computations, we get
\begin{eqnarray*}
 [u_{j,\pm}(\theta),u_{j,\pm}(\theta)]_J 
\!&=&\! \big(\omega_\theta\pm\lambda_{j,+}(\theta) \big)^2\!-\,s_j^2(\theta)\\
\!&=&\! 2 \lambda_{j,+}(\theta)\,\big( \lambda_{j,+}(\theta)\pm \omega_\theta\big). 
\end{eqnarray*}
From $\lambda^2_{j,+}(\theta)	-\omega_\theta^2=-s_j^2(\theta)$, we find that
\[
 [u_{j,+}(\theta),u_{j,+}(\theta)]_J\,[u_{j,-}(\theta),u_{j,-}(\theta)]_J\,
 =\,- 4 \,\lambda^2_{j,+}(\theta)\, s_j^2(\theta) \, <\, 0.
\] 
 Since  $Z$ is a $J$-unitary matrix, we have
 $$ [Zu_{j,\pm}(\theta),Zu_{j,\pm}(\theta)]_J =\, [u_{j,\pm}(\theta),u_{j,\pm}(\theta)]_J\, <\, 0.$$
Consequently, $Zu_{j,\pm}(\theta)$ is an eigenvector of $H_{\theta}(A)$ associated to  the eigenvalue $\lambda_{j,\pm}(H_\theta(A))$, such that one belongs to $\sigma_+^J(H_{\theta}(A))$, the other to $\sigma_-^J(H_{\theta}(A))$.
When $s_j(\theta)=0$, then (\ref{lambdaj}) reduce trivially to  $\Re(\e^{-i\theta}\alpha)$ in $\sigma_+^J(H_{\theta}(A))$ and to $\Re(\e^{-i\theta}\beta)$ in $\sigma_-^J(H_{\theta}(A))$.
\end{proof}

\medskip\smallskip
 
 We finish this section with some useful considerations. 
 If $\alpha=\beta$ in  Proposition \ref{lema1}, then the eigenvalues 
(\ref{lambdaj}) of $H_\theta(A)$ associated to a non-zero eigenvalue $\mu_j(\theta)$ of $M(\theta)$ are non-real and  $H_\theta(A)$ is not in class $\cal J$.  This is the reason for  assuming $\alpha \neq 
\beta$ in \eqref{bloco}, to avoid trivial cases of degeneracy of $W^J(A)$. 


The following observation is in order too. 
If $\Re(\e^{-i\theta}\alpha)$,  or $\Re(\e^{-i\theta}\beta)$,  is an eigenvalue   of $H_\theta(A)$, then  
the point $\alpha$, or $\beta$, respectively, is a component of $C^J(A).$ 
The remaining curves in $C^J(A)$ form a family which is central symmetric with respect to $\frac{1}{2}(\alpha+\beta)$ as implied by (\ref{lambdaj}). We are interested in imposing conditions on the blocks $C$ and $D$ that ensure that these curves,  whenever distinct from singletons,   are non-degenerate hyperbolas. Then the  boundary of $W^J(A)$ will be formed by  hyperbolic arcs and possibly flat portions, produced from the pseudo-convex hull of the different components of $C^J(A)$.
For that purpose, we will analyze  cases of matrices $M(\theta),$
 defined in (\ref{Mtheta}), whose eigenvalues can be explicitly computed. 


\section{Main Result}

Now, we assume that the non-diagonal blocks of the  matrix $A$ in (\ref{bloco}) are such that $DC$ is normal and commutes with $C^*C+DD^*$, so that they can be diagonalized by the same unitary similarity transformation,  inspired by the work \cite{Geryba}.
In this event, let us label the eigenvalues of $DC$ and $C^*C+DD^*$, respectively, by $z_j$ and $h_j$ according to the order in which they appear in the respective main diagonals. That is, letting  $U$ be the unitary matrix that simultaneously diagonalizes the above matrices, the $j$-th diagonal entries of 
$ U^*(C^*C+DD^*)\,U$ and 
$U^*DC U$
are  $h_j$ and $z_j$, respectively. 
In the sequel, for simplicity of notation, we use 
$$\Delta_j=(\alpha-\beta)^2+4z_j,$$
 and ${\cal \hat H}_j$ will denote the hyperbola with foci at 
$$
f_{j,\pm} =\,\frac{1}{2}(\alpha+\beta) \pm \frac{1}{2}\sqrt{\Delta_j}
$$
and transverse axis, non-transverse axis of length, respectively, equal to
\begin{equation}\label{eixos}
\! N_j=\! 
\left(\frac{1}{2}|\Delta_j|+\frac{1}{2}\left|\alpha-\beta\right|^2-h_j\right)^{\!\frac{1}{2}}\!, \quad  
M_j=\!
\left(\frac{1}{2}|\Delta_j|-\frac{1}{2}\left|\alpha-\beta\right|^2+h_j\right)^{\!\frac{1}{2}}\!. \
\end{equation}
As usual, ${\rm Arg}(z)$ denotes the principal argument of the complex number $z$. 


\medskip

\begin{theorem}\label{principal}  Let $A$ be of the block form {\rm (\ref{bloco})}, 
$n\leq 2r$, such  that $DC$ is normal and commutes with $C^*C+DD^*$, with  eigenvalues $z_j$ and $h_j$, $j=1, \dots, n-r$, respectively, labelled as mentioned above.  The following conditions are equivalent:
 \begin{itemize}
 \item[\bf (a)] $C^J(A)=\{{\cal \hat H}_1, \dots, {\cal \hat H}_{n-r}, \alpha\}$, whenever $n<2r$, or 
 $C^J(A)=\{{\cal \hat H}_1, \dots, {\cal \hat H}_{\frac{n}{2}}\}$, whenever $n=2r$;
\item[\bf (b)] For $j=1, \dots, n-r$, we have  \begin{equation}\label{cond hip 1}
|\alpha-\beta|^2-|\Delta_j| \, \leq  \,   2h_j \, <  \,
|\alpha-\beta|^2+ |\Delta_j|.
\end{equation} 
 \end{itemize}

\noindent In this case, $W^J(A)$ is the pseudo-convex hull of ${\cal \hat H}_1 \cup \cdots \cup {\cal \hat H}_{n-r}$. 
Moreover, each hyperbola ${\cal \hat H}_j$ is nondegenerate if and only if the corresponding LHS inequality in {\rm (\ref{cond hip 1})} is strict.
 
\end{theorem}

\begin{proof} 
Under the hypothesis, the matrix 
 $
M(\theta)
$
in \eqref{Mtheta} is unitary diagonalizable for all values of $\theta$ by the same unitary similarity as $DC$ and $C^*C+DD^*$, implying
that the eigenvalues of $M(\theta)$ are
\begin{equation}\label{eq1}
\mu_j(\theta)\, 
= \, h_j-2\,\Re(z_j)\cos(2\theta)-2\,\Im(z_j)\sin(2\theta), \qquad j=1, \dots, n-r.
\end{equation} 
By some computations, 
\begin{eqnarray*}
(\Re(\e^{-i\theta}\omega))^2 
&=&\big(\Re(\omega)\cos\theta+\Im(\omega)\sin\theta\big)^2\\
&=& (\Re(\omega))^2\cos^2\theta+(\Im(\omega))^2\sin^2\theta+\Re(\omega)\Im(\omega)\sin(2\theta)
\end{eqnarray*}
and
$$
(\Re(\omega))^2=\frac{\Re(\omega^2)}{2}+\frac{|\omega|^2}{2},
\quad (\Im(\omega))^2=-\frac{\Re(\omega^2)}{2}+\frac{|\omega|^2}{2}, \quad \Re(\omega)\Im(\omega)=\frac{1}{2}\Im(\omega^2).$$
Hence,
\begin{equation}\label{eq3}
 (\Re(\e^{-i\theta}\omega))^2\,=\,\frac{|\omega|^2}{2}+\frac{\Re(\omega^2)}{2}\cos(2\theta)+\frac{\Im(\omega^2)}{2}\sin(2\theta).
\end{equation}
Let $ \phi_j=\frac{1}{2}{\rm Arg}(\Delta_j)$.  We may conclude, by (\ref{eq1}) 
and (\ref{eq3}), that
\begin{eqnarray}
(\Re(\e^{-i\theta}\omega))^2-\frac{\mu_j(\theta)}{4}\!\! 
&=&\frac{|\omega|^2}{2}-\frac{h_j}{4}+\frac{\Re(\omega^2+z_j)}{2}\cos(2\theta)+\frac{\Im(\omega^2+z_j)}{2}\sin(2\theta)\nonumber \\
&=&\frac{1}{8}\left(4|\omega|^2-2h_j+|\Delta_j|\cos(2\theta-2\phi_j)\right) \nonumber \\ 
&=& \frac{1}{8}\left(4|\omega|^2-2h_j+|\Delta_j|-2|\Delta_j|\sin^2(\theta-\phi_j)\right).  \label{R>0}
\end{eqnarray}

Firstly, assuming {\bf (b)},
suppose that  \eqref{cond hip 1} holds with strict LHS inequality for some $j\in\{1, \dots, n-r\}$.
Then  $M_j$, $N_j$ given in \eqref{eixos} are positive and, in this case, we find that \eqref{R>0} is positive  for 
$\theta \in \Omega_j=\left(\phi_j -\eta_j, \phi_j +\eta_j\right)$,  considering $\eta_j=\arctan \left(N_j/M_j\right)$. 
 Consequently,  
$H_\theta(B)$ has 
real eigenvalues of the form
 $$
 \lambda_{j, \pm}\big(H_\theta(B)\big) \,=\, \pm\left(\frac{1}{4}N^2_j-\frac{1}{4}(M^2_j+N^2_j) \sin^2 (\theta -\phi_j)\right)^\frac{1}{2}, \qquad \theta \in \Omega_j,
 $$
according to Proposition \ref{lema1} and  the previous computations, 
 such that one of these eigenvalues belongs to $\sigma_+^J(H_{\theta}(B))$ and the other to $\sigma_-^J(H_{\theta}(B))$.

The equation of the 
tangent lines of a  component of $C^J({{\rm e}^{-i\phi_j}}B)$, perpendicular to the direction $\theta$, are of the form 
$$
x\cos\theta+y\sin\theta \,=\, \lambda_{j, \pm}\big(H_{\theta+\phi_j}(B)\big)\,,
$$
because $H_\theta ({{\rm e}^{-i\phi_j}}B)=H_{\theta+\phi_j} (B)$.
The envelope of this family of tangent lines, with
$\theta$ ranging over $\Omega_j$, gives this component in
$C^J({\rm e}^{-i\phi_j} B)$, which can be computed from its parametric equations 
$$
\left\{\begin{array}{ccc}
x &\!\! =\! \!&  \lambda_{j, \pm}\big(H_{\theta+\phi_j}(B)\big)\,\cos \theta \,-\, \lambda'_{j, \pm}\big(H_{\theta+\phi_j}(B)\big)\,\sin \theta\\
y &\!\! = \!\!&  \lambda_{j, \pm}\big(H_{\theta+\phi_j}(B)\big)\,\sin \theta \, +\, \lambda'_{j, \pm}\big(H_{\theta+\phi_j}(B)\big)\,\cos \theta
\end{array}\right.\!.
$$
We easily conclude  that this envelope   is a non-degenerate hyperbola centered at the origin (see the proof of Theorem 3.1 given in \cite[Theorem 2.1]{BLS_Hyp} for an analo\-gous reasoning); each branch of the hyperbola has the parametric equations associated to either $\lambda_{j,+}\big(H_{\theta+\phi_j}(B)\big)$ or $\lambda_{j, -}\big(H_{\theta+\phi_j}(B)\big)$. The transverse and  non-transverse axes have length $N_j$ and $M_j$.
 The corresponding rotated hyperbola in $C^J(B)= {\rm e}^{i \phi_j}  C^J({\rm e}^{-i \phi_j} B)$  has transverse axis parallel to $\e^{\phi_j}$ and the foci are $\pm\frac{1}{2}\sqrt{\Delta_j}$,
this implying, after applying a translation, that ${\cal \hat H}_j$ is a non-degenerate hyperbolic component  in  $C^J(A)$, 
because $A=B+\frac{1}{2}(\alpha+\beta)I_n$.

Moreover, if \eqref{cond hip 1} holds for $j=1, \dots, n-r$, then some components  of  $C^J(A)$  may (eventually) reduce to  points, this occuring if either
(i) the LHS inequality in \eqref{cond hip 1} holds as equality  for some $\iota\in\{1, \dots, n-r\}$ or if 
(ii)  $M(\theta)$ is singular. 
If (i)  holds, we have $M_{\iota }=0$ and $N_{\iota }>0$, then ${\cal\hat  H}_\iota$ reduces to the points $f_{\iota ,\pm}$,
one in $W_+^J(A)$, the other in $-W_-^J(A)$.
In case~(ii), we find $\alpha$  in $W_+^J(A)$ and $\beta$  in $-W_-^J(A)$ as points of $C^J(A)$. 
Otherwise, if $M(\theta)$ is non-singular and \eqref{cond hip 1} holds with strict LHS inequality for all $j =1, \dots, n-r$, then $C^J(A)$ contains at most $n-r$  nondegenerate concentric hyperbolas, ${\cal \hat H}_1, \dots,$ ${\cal \hat H}_{n-r}$, some of these possibly coincident (whenever $DC$ and $C^*C+DD^*$ have multiple eigenvalues).

Thus, we proved that if  {\bf (b)} is assumed, then  $C^J(A)=\{{\cal \hat H}_1, \dots, {\cal \hat H}_{n-r}\}$ and 
the point $\alpha$ is additionally in $C^J(A)$, if $n<2r$, that is, we find  {\bf (a)}. 
The reverse implication, {\bf (a)} implies {\bf (b)}, is clearly satisfied too.

In any of these cases,
 $W^J(A)$ is the pseudo-convex hull of ${\cal \hat H}_1 \cup \cdots \cup {\cal \hat H}_{n-r}$, as the point $\alpha$  is already included in this pseudo-convex hull.
\end{proof}

\medskip 
\smallskip

The main result in Theorem \ref{principal} can be easily adapted to the  case $n>2r$, just interchanging the roles of $C,D$, together with those  of $\alpha, \beta$ and $r, n-r$.  

The conditions imposed in 
 Theorem \ref{principal}, concerning  the non-scalar blocks $C$, $D$ of the  block matrix  $A$ of the form (\ref{bloco}),
are essential to obtain hyperbolic components in $C^J(A)$, as the next example shows. 

\smallskip

\begin{Exa} \label{exdelto} Let $J=I_2\oplus -I_2$ and $A$ be the block matrix of form {\rm (\ref{bloco})}, with 
$$ 
\alpha=1,\qquad \beta=-7, \qquad 
 C=\left[
\begin{array}{cc}
 2 & 0 \\
 3 & 3 \\
\end{array}
\right], \qquad  D=\left[
\begin{array}{cc}
 3 & 2 \\
 3 & 0 \\
\end{array}
\right].
$$
Then $$
C^*C+DD^*=\left[
\begin{array}{cc}
 26 & 18 \\
 18 & 18 \\
\end{array}
\right] \qquad \hbox{and}\qquad
DC=\left[
\begin{array}{cc}
 12 & 6 \\
 6 & 0 \\
\end{array}
\right].
$$
The matrix $DC$ is normal, but it does not commute with $C^*C+DD^*$, so the hypothesis of Theorem \ref{principal} are not fullfilled.
In this case,  $C^J(A)$ consists of a hyperbolic-like shaped curve (which is not a hyperbola) and two deltoids in its interior (see Figure \ref{Fig1}).

\begin{figure}[!h]
\centering
\includegraphics[scale=0.75]{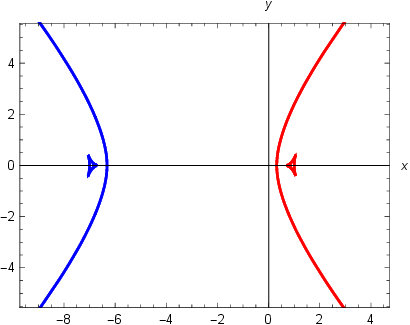}
\caption{Boundary generating curve of $W^J(A)$ in Example \ref{exdelto}}\label{Fig1}
\end{figure}
\end{Exa}

\section{Arcs of hyperbolas and flat portions on the boundary of $W^J(A)$}

\smallskip

Under the hypothesis of Theorem \ref{principal}, if $C^J(A)$ has different hiperbolic components, with points on the boundary of $W^J(A)$,
then boundary flat portions will appear too. 
These boundary line segments  are naturally central-symmetric with respect to the point $\frac{1}{2}(\alpha+\beta)$. 
In other cases,  the shape of $W^J(A)$ may be hyperbolic, 
as discussed in  detail in Section 5.

Example \ref{ex9} illustrates the existence of boundary flat portions that intersect at a corner of $W^J(A)$, 
that is, a  boundary point of  $W^J(A)$  on more than a support line.
As it is well known, any corner of $W^J(A)$ is an eigenvalue of $A$ \cite{LR PAMS}. 

In all the figures used throughout to illustrate, the branches of the hyperbolas, as well as the points, in $W_+^J(A)$ are represented in blue and those in $-W_-^J(A)$ are colored in red.

\begin{Exa} \label{ex9} Let $J=I_3\oplus -I_3$ and $A$ be a block matrix of the form {\rm (\ref{bloco})}, with 
$$ \alpha=10,\qquad \beta=6, \qquad
 C=\left[
\begin{array}{ccc}
 1 & 0 & 1 \\
 0 & 1 & 1 \\
 1 & 1 & 0 \\
\end{array}
\right],\qquad
 D=\left[
\begin{array}{ccc}
 0 & 1 & 1 \\
 1 & 0 & 1 \\
 1 & 1 & 0 \\
\end{array}
\right].
$$
Then $C,D$ 
and $$
DC=\left[
\begin{array}{ccc}
 1 & 2 & 1 \\
 2 & 1 & 1 \\
 1 & 1 & 2 \\
\end{array}
\right]=CD
$$
are real symmmetric and $DC$ commutes with $C^2+D^2=2D+4I_2$, that is, the hypothesis of Theorem \ref{principal} are satisfied. Since  $h_1=8$, $h_2=h_3=2$, $z_1=4$, $z_2=1$, $z_3=-1$, then {\rm (\ref{cond hip 1})} holds for $j=1,2,3$,  with equality at the  LHS of {\rm (\ref{cond hip 1})} for $j=2$. Therefore, 
$C^J(A)$ consists of  the points $8\pm\sqrt{3}$ and two nested hyperbolas, 
 with foci $8\pm \sqrt{8}$ and $8\pm\sqrt{5}$, respectively,   depicted in Figure~2.
The set  $W^J(A)$ is the pseudo-convex hull of the  outer equilateral hyperbola, with foci  at $8 \pm \sqrt{8}$ and axes of length $4$, and the points 
$ 
8+\sqrt{3}\,\in W_+^J(A)$, $8-\sqrt{3}\in -W_-^J(A),
$ 
which are corners of $W^J(A)$
(see Figure \ref{Fig3}).  

\begin{figure}[!h]
	\centering 
	\begin{minipage}[t]{5.8cm} 
	\centering 
	\includegraphics[scale=0.58]{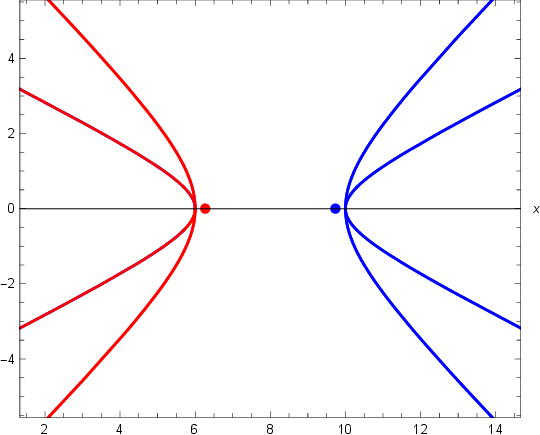}  
\caption{$C^J(A)$ in Example \ref{ex9}} \label{Fig2}
	\end{minipage} 
	\begin{minipage}[t]{6.2cm} 
		\centering 
	\includegraphics[width=\columnwidth]{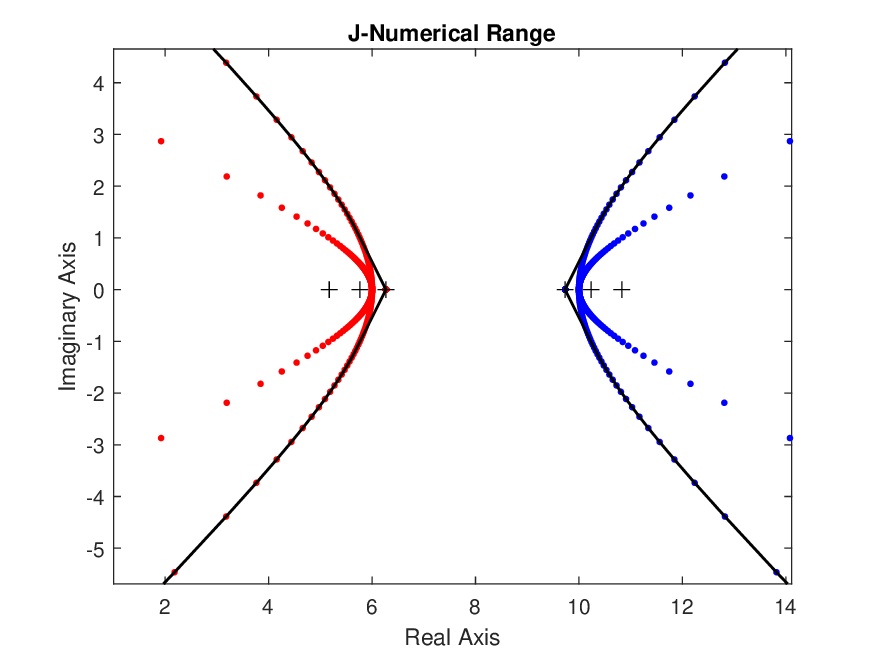} 
			\caption{$W^J(A)$ in Example \ref{ex9}} \label{Fig3}
	\end{minipage} 
\end{figure}	
\end{Exa}

\smallskip

Next, Example \ref{ex1} illustrates Theorem \ref{principal} in the case when non-nested hyperbolas are present in $C^J(A)$.

\medskip

\begin{Exa} \label{ex1} 
Let $J=I_3\oplus -I_3$ and $A$ be of the block form {\rm (\ref{bloco})} with 
$$\alpha=3,\quad \beta=-3, \quad 
C=\left[
\begin{array}{ccc}
 2+2 i & 1-i & 0 \\
 -i & -1+i & 0 \\
 0 & 0 & 4 \\
\end{array}
\right], \quad D=\left[
\begin{array}{ccc}
 i & 0 & 0 \\
 i & 3+i & 0 \\
 0 & 0 & \frac{1}{4} \\
\end{array}
\right].
$$
Then
$$  
C^*C+DD^*=\left[
\begin{array}{ccc}
 10 & -5 i & 0 \\
 5 i & 15 & 0 \\
 0 & 0 & \frac{257}{16} \\
\end{array}
\right],\qquad
DC=\left[
\begin{array}{ccc}
 -2+2 i & 1+i & 0 \\
 -1-i & -3+3 i & 0 \\
 0 & 0 & 1 \\
\end{array}
\right],
$$ 
such that $DC$ is normal 
and comutes with 
$C^*C+DD^*$. 
After some compuations, we can confirm that {\rm (\ref{cond hip 1})} holds, with strict LHS inequality, for $j=1,2,3$. By Theorem \ref{principal},
$C^J(A)$  consists of  three hyperbolas, all centered at the origin, as displayed in Figure 4. 
The boundary of $W^J(A)$ contains hyperbolic arcs, from the two non-nested hyperbolas 
and consequently 
each convex set $W_+^J(A)$ and $-W_-^J(A)$ will have a line segment on its boundary.

 \begin{figure}[!h]\label{F1}
\centering
\includegraphics[scale=0.7]{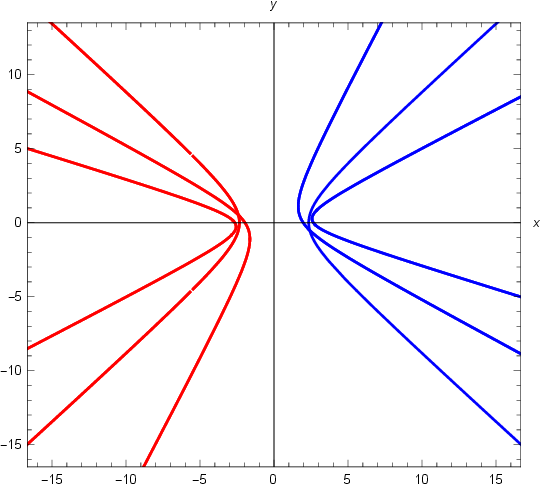}
\caption{Boundary generating curve of $W^J(A)$ in Example \ref{ex1}}
\end{figure}

\end{Exa}

Theorem \ref{principal} holds, in particular, for matrices of the form  {\rm (\ref{bloco})} when $CD$ and $DC$ are both normal. The following result was obtained in \cite[Theorem~3.1]{BLPS}  and is now proved as an easy consequence of
Theorem \ref{principal}. In this case, non-nested hyperbolas may be present in $C^J(A)$ too. 


\smallskip

\begin{corol} \label{Rmk DC CD COMUT} 
Let $A$ be of type {\rm (\ref{bloco})}, such that $CD$ and $DC$ are both normal,  $p=\min\{r,n-r\}$. Let  $\sigma_1, \dots, \sigma_{n-r}$ and $\delta_1, \dots, \delta_{r}$  be the  singular values of $C$ and $D$, respectively. If
\begin{equation}\label{desigual DC CD comu}
2\Re (\overline{f_{j+}}f_{j-}) \, < \, |\alpha|^2 + |\beta|^2-\sigma_j^2 - \delta_j^2 \, < \, |f_{j+}|^2 + |f_{j-}|^2, \quad j=1,\dots, p, \
\end{equation}
then  $C^J(A)$ has at most $p$ hyperbolic components  (some possibly coincident), with foci at $f_{j\pm}$ and
 non-transverse axis of length 
$$
\sqrt{|f_{j+}|^2 +|f_{j-}|^2-|\alpha|^2-|\beta|^2+\sigma^2_j+\delta_j^2}, \qquad j=1,\ldots,p,
$$
and possibly a point, $\alpha$ if $n<2r$ and $\beta$ if $n>2r$.
The set $W^J(A)$ is the pseudo-convex hull of these $p$ hyperbolas.
\end{corol}

\begin{proof} Without loss of generality, let $n\leq 2r$.  The hypothesis on $CD$ and $DC$ being both  normal is equivalent to the existence of unitary matrices $U,V$ such that $U^*C^*V=\Sigma$, $U^*D\,V=\Gamma$ are
 both dia\-gonal {\rm \cite[p. 426]{HJ}}. This implies that 
 $$
  U^*(C^*C + DD^*) U =\Sigma \Sigma^* +\Gamma\Gamma^*   \qquad \hbox{and} \qquad 
  U^*DC\,U =\Gamma \Sigma^*
 $$ 
are  diagonal matrices too, that is, $C^*C+DD^*$ and $DC$  are simultaneously unitarily diagonalizable, 
 with  eigenvalues  $h_j = \sigma^2_j+\delta_j^2$
and  $z_j=\sigma_j\delta_j \e^{i\chi_j}$, for some $\chi_j\in \mathbb{R}$, for $j=1, \dots, n-r$, respectively. 
Then we are under the hypothesis of  Theorem \ref{principal}.
Moreover, 
$$
2\Re (\overline{f_{j+}}f_{j-}) =\, \frac{1}{2}|\alpha+\beta|^2 - \frac{1}{2}|\Delta_j|, \qquad |f_{j+}|^2+|f_{j-}|^2 =\, \frac{1}{2}|\alpha+\beta|^2 + \frac{1}{2}|\Delta_j|. 
$$
Thus, inequalities (\ref{desigual DC CD comu}) are equivalent to 
$$
\big|\,|\alpha-\beta|^2 -  2\sigma^2_j-2\delta_j^2\, \big|\, <  \, |\Delta_j| 
\qquad j=1, \dots, n-r.
$$
By Theorem \ref{principal}, the result follows.
\end{proof}

\medskip
\smallskip

As illustrated by the matrices in Example \ref{ex1} 
 the main result in Theorem~\ref{principal} holds even if $CD$ is not normal, that is, it extends Corollary \ref{Rmk DC CD COMUT}. 
In fact, for $C,D$ as in Example \ref{ex1}, we have that
$$
CD=\left[
\begin{array}{ccc}
 -1+3 i & 4-2 i & 0 \\
 -i & -4+2 i & 0 \\
 0 & 0 & 1 \\
\end{array}
\right]
$$
is not normal and Corollary \ref{Rmk DC CD COMUT} does not apply.


In particular, Corollary \ref{Rmk DC CD COMUT} materializes for $A\in M_n$ of the form (\ref{bloco}),  $n=2r$, when $C=U-I_r$, $D=U^*+I_r$ and $U \in M_r$ is a unitary matrix. Then  $CD=DC=2i\,\Im (U)$ is skew-Hermitian and $C^*C+DD^*=4I_{r}$, corresponding to $h_j=4$ and to $z_j$ zero or pure imaginary with $|z_j|\leq 2$, $j=1, \dots, r$,  in Theorem \ref{principal}.

\medskip

\begin{Exa} \label{ex10} Let $J=I_3\oplus -I_3$ and $A$ be of the form {\rm (\ref{bloco})} with $\alpha=10,$ $\beta=7$,  $C=U-I_4$, $D=U^*+I_4$, considering the unitary matrix
$$
U\,=\left[
\begin{array}{ccc}
 \frac{1}{2}+\frac{i}{2} & \frac{1}{2}-\frac{i}{2} & 0 \\
 \frac{1}{2}-\frac{i}{2} & \frac{1}{2}+\frac{i}{2} & 0 \\
 0 & 0 & -i \\
\end{array}
\right].
$$
In this case, $DC$ has eigenvalues $2i,0,-2i$ 
and  $C^J(A)$ consists of three   non-nested co-centered hyperbolas, depicted in Figure 5. 
The pseudo-convex hull of these  hyperbolas yields a vertical flat portion  on the boundary of $W^J(A)$.

\begin{figure}[!h]\label{Fvert}
\centering
	\includegraphics[scale=0.67]{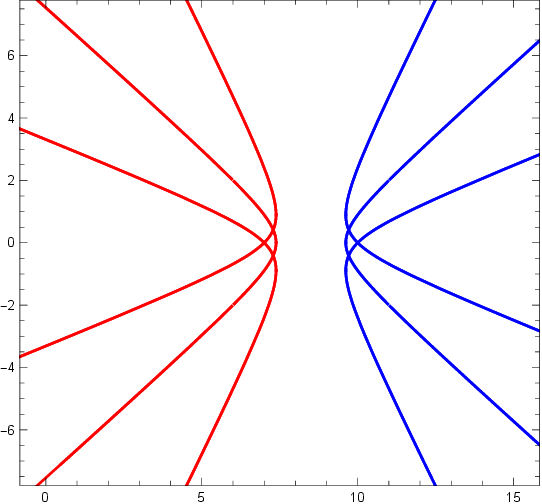} 
\caption{$C^J(A)$ in Example \ref{ex10}}
\end{figure}	
\end{Exa}


In general, it is known that $W^J(A)$ contains a flat portion on the boundary, if   $H_\theta(A)$ has a multiple eigenvalue $\lambda$, for  some $\theta \in \mathbb{R}$,
that is a maximum or a minimum of 
$\sigma^J_+(H_\theta (A))$ or $\sigma^J_-(H_\theta (A))$, 
 $\lambda=\lambda_R(\theta )$ or $\lambda=\lambda_L(\theta )$, 
 and the set 
 of numbers $[H_{\theta+ \frac{\pi}{2}}(A) x,x]_J$, when $x$ ranges over the eigenspace $\cal E_\lambda$ of $H_\theta (A)$ associated to $\lambda$, 
 is not a singleton.
In this event, 
if $x_1, x_2$ are eigenvectors in $\cal E_\lambda$, satisfying $[x_1,x_1]_J[x_2,x_2]_J>0$ and
 $$[H_{\theta+\frac{\pi}{2}}(A)\,x_1, x_1]_J / [x_1, x_1]_J \, \neq\,  [H_{\theta+\frac{\pi}{2}}(A)\,x_2, x_2]_J  / [x_2,x_2]_J,$$ 
then the boundary of 
$W^J(A)$ contains the line segment joining $[Ax_1,x_1]_J$ and $[Ax_2, x_2]_J$. 

\section{Classes of matrices with hyperbolic $W^J(A)$}

\smallskip

In this section, 
 selected classes of block matrices $A$ of the form {\rm (\ref{bloco})} with hyperbolic shaped $W^J(A)$ are presented.
The next corollary  applies to  Hermitian matrices $A$ of these block form.

\medskip

\begin{corol}\label{Hermitica}
Let $A\in M_n$ be non-diagonal of the form {\rm (\ref{bloco})} with  $\alpha, \beta \in \mathbb R$ and $D=C^*$.
Let $\sigma_1>\cdots >\sigma_s$ be the non-zero singular values of $C$.  
Then  $C^J(A)$ consists of $s$ nested hyperbolas ${\tilde H}_1, \dots, {\tilde H}_s$ and, additionally, the points $\alpha, \beta$, if $C$ is not full rank, the point $\alpha$ if $n<2r$, and $\beta$ if $n>2r$.
Each  hyperbola ${\tilde H}_j$ is centered at $\frac{1}{2}(\alpha+\beta)$, has horizontal transverse semi-axis of lenght $|\omega|$ and vertical non-transverse semi-axes of length $\sigma_j$. The set $W^J(A)$ is the non-degenerate hyperbolic disc bounded by ${\tilde H}_1$.

\end{corol}

\begin{proof}  Firstly, let $n\leq 2r$. Clearly, $A$ is under the conditions of  Theorem~\ref{principal},   $h_j=2\sigma^2_j$, $z_j=\sigma^2_j$, $j=1, \dots, s$,  and {\rm \eqref{cond hip 1}} holds for $j=1, \dots, s$, with strict LHS inequality.
Hence, if $n=2r$, then $C^J(A)$ consists of $s$ non-degenerate hyperbolas corresponding to the non-zero singular values of $C$ and the points $\alpha, \beta$, if $C$ is not full rank. If $n<2r$, then  $C^J(A)$ consists
of these hyperbolas and the point $\alpha$. By Theorem~\ref{principal}, the hyperbolas have foci at 
$$
\frac{1}{2}(\alpha+\beta) \pm \sqrt{\omega^2+\,\sigma^2_j}, \qquad  j=1,\dots,s,
$$
all have the same transverse semi-axis of lenght $|\omega|$ and non-tranverse semi-axis of lenght $\sigma_j$, $j=1,\dots,s$, respectively. Thus, they are nested 
and  $W_J(A)$ is the non-degenerate hyperbolic disc bounded by the outer hyperbola.

If $n>2r$, replacing $C$ by $D$ 
 and $\alpha$ by $\beta$ 
leads to the same conclusion, because $C,C^*$ have the same non-zero singular values.
\end{proof}

\medskip
\smallskip

Theorem \ref{principal} holds, in particular, when $DC$ is a scalar multiple of the identity, in which case all $z_j$ are the same.
 The introduced notation  for the hyperbola ${\cal \hat H}_j$ and its foci $f_{j, \pm}$  from Section 3
is reused below.

\smallskip

\begin{corol}\label{corol DC escalar}  Let $A$ be of type {\rm (\ref{bloco})}, $DC = z_1 I_{n-r}$, 
$n\leq 2r$.
The set $W^J(A)$ is a non-degenerate hyperbolic disc with foci at $f_{1\pm}$ and non-transverse axis  of length $M_1$ if and only if 
\begin{equation}\label{condz1}
 \left||\alpha-\beta|^2-2\|C^*C+DD^*\|\right| \,<\, \left|(\alpha-\beta)^2+4z_1\right|.
\end{equation}
In this case, all the hyperbolic components $\hat{\cal H}_1,\dots,\hat{\cal H}_{n-r}$ in $C_J(A)$ are nested.
\end{corol}

\begin{proof} 
Under the hypothesis, the hyperbolas $\hat{\cal H}_1,\dots,\hat{\cal H}_{n-r}$ share the foci $f_{1\pm} $, 
and the axes lenghts $N_j, M_j$  in (\ref{eixos}) of  $\hat{\cal H}_j$ vary just according to the eigenvalue $h_j$ of $C^*C+DD^*$. Therefore, all the hyperbolas are nested. If $h_1=\|C^*C+DD^*\|$, 
then $N_1$ is the minimum transverse axis and, consequently, $M_1$ is the maximum  non-transverse axis. The result readily follows from Theorem
\ref{principal}.
\end{proof}

\smallskip

\begin{Exa} \label{ex5} Let $J=I_3\oplus -I_2$ and $A$ be of the form {\rm (\ref{bloco})} with 
$$\alpha=6+i,\qquad \beta=13, \qquad 
C=\left[\begin{array}{cc}
-2 & -1\\
-4 & -2\\
-3 & 0
\end{array}\right], \qquad D=\left[\begin{array}{ccc}
0 & 0 & 2\\
0 & 3 & -4
\end{array}\right].
$$
Then $DC=-6I_2$,
$$
C^*C+DD^*=\left[\begin{array}{cc}
33 & 2\\
2 & 30
\end{array}\right], \qquad \|C^*C+DD^*\|=34.
$$  
By simple computations, {\rm (\ref{condz1})} holds and, by Corollary \ref{corol DC escalar}, the set $W_J(A)$ is the hyperbolic disc, with foci at 
$$\frac{1}{2}(19+i)\pm \frac{1}{2}\sqrt{24-14\,i}$$
and non-transverse axis of length  
$\left(9+\sqrt{193}\right)^{\frac{1}{2}}$. 
By Theorem \ref{principal}, $C^J(A)$ consists of two co-centered hyperbolas and a point, $(6,1)$, displayed in Figure 6. 

 \begin{figure}[!h]\label{ExDCescalar}
\centering
\includegraphics[scale=0.7]{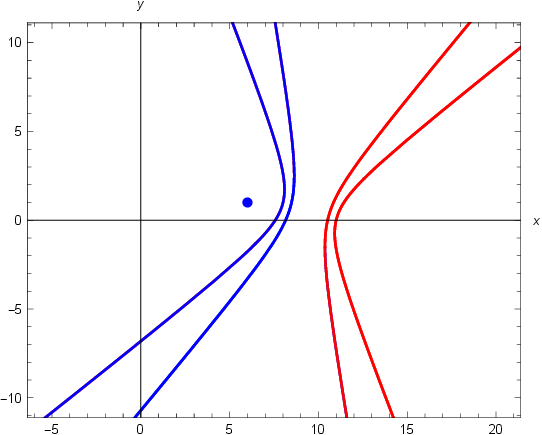}
\caption{Boundary generating curve of $W^J(A)$ in Example \ref{ex5}}
\end{figure}


\end{Exa}

A {\it quadratic matrix} is a matrix with minimal polynomial of degree $2$
 and it can be reduced, under a unitary similarity transformation, to a matrix $A$ of type \eqref{bloco} with $D=O$.
In this case, we easily find the next result (cf.\! \cite{BLPS, ELA quadr}).

\medskip

\begin{corol} Let $A$ be of type {\rm (\ref{bloco})} with  $C\neq O$ and  $D = O$. 
Then $W^J(A)$ is a non-degenerate hyperbolic disc, with
foci at $\alpha$ and $\beta$, and non-transverse axis of lenght $\|C\|$ if and only if $\|C\|<|\alpha -\beta|$.
\end{corol}

\begin{proof} Since  $D=O$, the result follows from Corollary \ref{corol DC escalar} with $z_1=0$. In this case,   
$f_{1\pm}$ reduce to $\alpha$, $\beta$. Further, $\|C^*C\|=\|C\|^2\neq 0$ and  the  non-degenerate hiperbolicity condition (\ref{condz1}) of $W^J(A)$ becomes $
\left||\alpha-\beta|^2-2\|C\|^2\right| \,<\, \left|\alpha-\beta\right|^2,
$ 
which can equivalently be simplfied to $\|C\|<|\alpha -\beta|$.
\end{proof}
 
 \medskip\medskip
	
If $r=n-1$, we have $J=I_{n-1}\oplus -I_1$ and the matrix  in (\ref{bloco}) reduces to an {\it arrowhead matrix} of type
\begin{equation}\label{arrow}
\left[
\begin{array}{ccc|c}
\!\alpha &  &  & c_1 \\
     & \! \ddots &  & \vdots \\
     &    &   \! \alpha & c_{r} \\ \hline
d_1  &  \cdots & \! d_{r} & \beta
\end{array}
\right],
\end{equation}
with zeros at the omitted entries of the first diagonal block. 
The next result holds, in parallel to the classical numerical range elliptical case~\cite{Linden, Chien}. 

\medskip
 
\begin{corol}\label{arrow1} For an arrowhead matrix  $A\in M_n$ of type  {\rm (\ref{arrow})}, $r=n-1$,
${\bf c}=(c_1, \ldots, c_{r})$, ${\bf d}=(d_1, \ldots, d_{r})$,
the set $W^J(A)$ is a non-degenerate hyperbolic disc, 
with foci at 
$$\frac{1}{2}(\alpha+\beta)\pm \frac{1}{2}\sqrt{(\alpha-\beta)^2+ 4\,{\bf c}^T{\bf d}},$$
 non-transverse axis
of length
$$\left( \frac{1}{2}\left|(\alpha-\beta)^2 + 4\,{\bf c}^T {\bf d}\right|-\frac{1}{2}|\alpha-\beta|^2 +\|{\bf c}\|^2+  \|{\bf d}\|^2 \right)^{\!\frac{1}{2}},$$ 
if and only if
$$
\left||\alpha-\beta|^2-2\|{\bf c}\|^2 - 2\|{\bf d}\|^2\right| \,<\, \left|(\alpha-\beta)^2 + 4\,{\bf c}^T {\bf d}\right|.
$$
\end{corol}

\begin{proof}
The non-diagonal blocks of the arrowhead matrix $A$ are such that
$$
 DC\,=\,c_1d_1+\cdots+c_{r}d_{r}={\bf c}^T{\bf d},
$$  $$C^*C+DD^*= |d_1|^2+\cdots+ |d_{r}|^2+|c_1|^2+\cdots+ |c_{r1}|^2 = \|{\bf c}\|^2+ \|{\bf d}\|^2$$
 and the result is a direct consequence of Corollary \ref{corol DC escalar}. 
\end{proof}

\medskip



\medskip

Now, we consider the case when the spectrum of $M(\theta)$ is independent of $\theta$
and study the condition for non-degenerate hyperbolicity of  $W^J(A)$.

\smallskip

\begin{theorem}\label{Teorema2}  Let $A\in M_n$ be a block matrix of the form {\rm (\ref{bloco})} with 
 the spectrum of $M(\theta)$ in \eqref{Mtheta} independent of $\theta$.
The set $W^J(A)$ is a hyperbolic disc, 
with foci at $\alpha$ and $\beta$, 
 non-transverse axis of length $\|C-D^*\|$,
 if and only if  $$\|C-D^*\|<|\alpha-\beta|.$$
In this case,  all the hyperbolic components of $C^J(A)$ are nested,  with the same foci,  non-transverse axes of lenght    
$\sqrt{\mu_1}, \dots, \sqrt{\mu_s}$, where $\mu_1,\dots,\mu_s$ are the distinct non-zero eigenvalues  of $M(0)$.
\end{theorem}

\begin{proof} By hypothesis, the spectrum of  $M(\theta)$ is  independent of $\theta$, thus it coincides with the spectrum of  $$M(0)\,=\, C^*C+DD^*-2\Re(DC)=(C-D^*)^*(C-D^*).$$ 
Suppose that the maximum eigenvalue of $M(0)$ is $\mu_1$, that is, $\mu_1=\|C-D^*\|^2$.
If $\|C-D^*\| < |\alpha-\beta|$ is assumed  and $\theta \in \Omega'_j=\left(\tau-\theta_j,\tau +\theta_j\right)$, where 
$$\tau={\rm Arg}(\alpha-\beta),
 \qquad \theta_j=\arctan \left(\frac{|\alpha-\beta|^2}{\mu_j}-1\right)^{\!\frac{1}{2}},$$
then $$  \mu_j \, < \, \big(\Re(\e^{-i\theta}(\alpha-\beta))\big)^2 \, = \, |\alpha-\beta|^2\cos^2(\theta-\tau), \qquad j=1, \dots, s,$$
and it follows from Proposition \ref{lema1} that  $H_\theta(B)$ is in class $\cal J$, with $s$ pairs of real eigenvalues of the form 
  $$
\lambda_{j, \pm}\big(H_\theta(B)\big)\,=\pm \frac{1}{2}\left(|\alpha-\beta|^2-\mu_j-|\alpha-\beta|^2\sin^2\!\left(\theta-\tau\right)\right)^{\frac{1}{2}}, \quad j=1, \dots, s.
  $$
We conclude that $C^J(B)$ contains $s$  hyperbolic components, all with the same foci $\pm \frac{1}{2}(\alpha-\beta)$ and 
non-transverse axes of lenght 
$\sqrt{\mu_j},$ 
$j=1, \dots, s.$
 This corres\-ponds to  $s$ hyperbolic nested components in $C^J(A)$, all of them with foci at  $\alpha$ and $\beta$, 
with the previous non-transverse axes lenghts.

 If $M(0)$ is not full rank, or $n\neq 2r$, then the remaining eigenvalues of $H_\theta (B)$ would additionally give 
 $\alpha$ and $\beta$, or just one of these,  as points in $C^J(A)$.
Clearly, $W^J(A)$ is bounded by the outer hyperbola, obtained for $j=1$, with non-transverse axis of length
$\|C-D^*\|$.  


 
 
The reverse implication is trivially satisfied. 
 \end{proof}


\medskip\medskip

If  for a subspace $S$ invariant under $C^*C+DD^*$, the column space
of $DC$ is contained in $S$ and  $S$ is contained in the 
null space of $DC$, then the eigenvalues  of {\rm (\ref{Mtheta})}  are independent of $\theta$, as observed in {\rm \cite{Geryba}} and
then Theorem \ref{Teorema2} holds.

\smallskip


\begin{Exa} \label{ex7} Let $J=I_2\oplus -I_2$ and $A$ be the block matrix of the form {\rm (\ref{bloco})} with 
$$
 \alpha=8+i,\quad \ \beta=4-i,\quad  \
 C=\left[\begin{array}{cc}
1-i & 1-i\\
-1 & 2
\end{array}\right], \ \quad
 D=\left[\begin{array}{cc}
1+i & 2\\
0 & 0
\end{array}\right].
$$
Then $$
DC=\left[\begin{array}{cc}
0 & 6\\
0 & 0
\end{array}\right] \qquad \hbox{and}\qquad
C^*C+DD^*=\left[\begin{array}{cc}
9 & 0\\
0 & 6
\end{array}\right].
$$
The matrix  in \eqref{Mtheta} becomes 
$$
M(\theta)=\left[
\begin{array}{cc}
 9 & -6\, \e^{-2 i \theta } \\
 -6\, \e^{2 i \theta } & 6 
\end{array}
\right],
$$
which has eigenvalues $\frac{15}{2}\pm\frac{3}{2}\sqrt{17}$  independent of $\theta$. By Theorem \ref{Teorema2}, the set $W^J(A)$ is a non-degenerate hyperbolic disc 
with foci at $(8,1)$ and $(4,-1)$, non-transverse axis of length $$\sqrt{\frac{15}{2}+\frac{3}{2}\sqrt{17}}\,;$$
and there are two nested hyperbolic components in $C^J(A)$ as Figure 7 shows.

\begin{figure}[!h]
\centering
\includegraphics[scale=0.7]{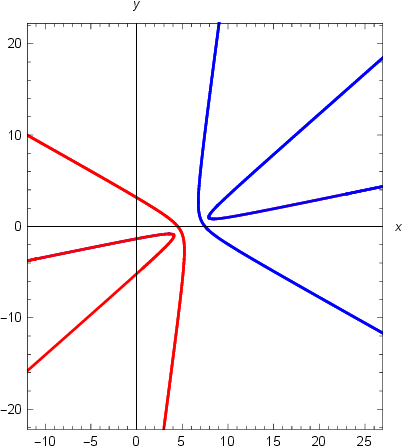}
\caption{Boundary generating curve of $W^J(A)$ in Example \ref{ex7}}
\end{figure}
\end{Exa}


\section{Results for certain tridiagonal matrices}

\medskip

In this section, we  consider ${\tilde J}={\rm diag}(1,-1,1,-1,\dots)\in M_n$  and  focus on the class of tridiagonal matrices with  biperiodic main diagonal  of the form $(\alpha,\beta,\alpha,\beta, \dots)\!\in{\mathbb C}^n$,  denoted by
  $T_{n}(\alpha, \beta; {\bf c},{\bf d})$, where ${\bf c}=(c_1,\ldots,c_{n-1})$ and ${\bf d}=(d_1,\ldots, d_{n-1})\in {\mathbb C}^{n-1}$ 
are, res\-pectively,  the first  lower and first upper subdiagonals of the matrix.

By  \cite[Lemma 2.1]{ELA},  the ${\tilde J}$-numerical range of these tridiagonal matrices is invariant under  switching of any two corresponding off-diagonal entries  $c_j$ and $d_j$.
Consider $K \subset \{1,\dots,n-1\}$ and ${\tilde T}$ the matrix obtained from  $T_{n}(\alpha, \beta; {\bf c},{\bf d})$ by interchanging  $c_j$ and $d_j$ for  $j\in K$.
Let  $r=\ceil[bbig]{\frac{n}{2}}$  and  $P$ be the permutation matrix associated with the permutation $\pi\in S_n$ defined by
\begin{equation}\label{permut}
   \pi (i)=2i-1, \quad 1 \leq i \leq r,\qquad \hbox{and}\qquad \pi \left(r+i\right)=2i, \quad 1 \leq i \leq n-r.
\end{equation}
Clearly, 
 ${\tilde J} = P^T J P$ and ${\tilde T}=P^TA\,P$ for
$J=I_r\oplus-I_{n-r}$ and  $A$ a block matrix of type (\ref{bloco}), where $n=2r$ or $n=2r-1$. Therefore, $W^{\tilde J}({\tilde T})=W^{J}(A)$ and the results of the previous sections may be applied to conveniently derive corresponding characterizations of the indefinite Kippenhahn curve and ${\tilde J}$-numerical range of tridiagonal matrices of type $\tilde T$.
 
The next corollary exemplifies this situation.
It was obtained in \cite[Theorem~3.1]{BLS_Hyp} when ${\bf a}=(a, -a, a, -a, \dots )$, $a\in \mathbb R$ (see also \cite{ELA}). 

For simplicity, let
$\omega=\frac{1}{2}(\alpha -\beta)$ 
and consider the bidiagonal matrix $X({\bf c})$ defined,  in terms of the entries of  the vector ${\bf c}=(c_1, \dots, c_{n-1})$,  by either
\begin{equation}\label{biD}
\left[\begin{array}{ccccc}
c_1  &   & &  &\\
c_2 & c_3 &   &  &  \\
 & c_4 & c_5 &   &  \\
 &  & \ddots & \ddots &  \\
 &   &   &  c_{n-2} & c_{n-1}
\end{array}\right] \quad \ \hbox{or} 
\quad \ \left[\begin{array}{cccc}
c_1  &  &  &  \\
c_2 & c_3 &  & \\
 & \ddots &  \ddots  & \\
 &  &   c_{n-3} & c_{n-2}\\
 &  &    & c_{n-1} 
\end{array}\right],
\end{equation}
according to $n$ even or $n$ odd, with zeros at the remaining (ommitted) entries.
  
\medskip\smallskip

\begin{corol}  Let $T=T_{n}(\alpha, \beta;  {\bf c},{\bf d})$ 
with ${\bf d}=\kappa\,\overline{{\bf c}}$, $\kappa \in \mathbb C$.
Consider $s_1 \geq \ldots \geq s_{\lfloor \frac{n}{2}\rfloor}$ the singular values 
and $k$ the rank of  $X({\bf c})$. 
Let ${\cal H}_j$ be the hyperbola
with foci at 
$$ 
\frac{1}{2}(\alpha+\beta) \pm  \sqrt{\omega^2 +\kappa\,s^2_j\,} 
$$ 
and non-transverse axis of length 
$$ 
  \sqrt{2\,\big|\omega^2+\kappa\,s_j^2\big|-2\,|\omega|^2+(|\kappa|^2+1)s_j^2}.
$$ 
Then $C^J(T)$ consits of the nested hyperbolas ${\cal H}_1, \dots, {\cal H}_k$ and possibly the points $\alpha,\beta$, if $k<\floor[big]{\frac{n}{2}}$,  or the point $\alpha$, if $n$ is odd and $k=\floor[big]{\frac{n}{2}}$, 
 if and only if
\begin{equation}\label{condTrid1}
 \left||\,\omega|^2-\frac{|\kappa|^2+1}{2}s_1^2\right|\, <\,  \big|\,\omega^2+\kappa\,s_1^2 \big|.
\end{equation}
 In this case,  $W^{\tilde J}(T)$ is 
bounded by the non-degenerate hyperbola ${\cal H}_1$.
 \end{corol}

\begin{proof}  Let  ${\tilde T}=T_n(\alpha, \beta;\widetilde{\bf c},\widetilde{\bf d})$, $\widetilde{\bf c}=(c_1,\kappa \overline{c_2},c_3,\ldots)$ and $\widetilde{\bf d}=(\kappa  \overline{c_1},c_2,\kappa \overline{c_3},\ldots)$. 
Considering
$J=I_r\oplus-I_{n-r}$, $r=\ceil[bbig]{\frac{n}{2}}$, and 
$A$ the block matrix in (\ref{bloco}) with $C=X({\bf c})$ and $D=\kappa\,C^*$, having in mind the previous considerations, 
we may conclude that $$W^{\tilde J}(T)\,=\,W^{\tilde J}(\tilde {T})\,=\, W^{J}(A).$$ 
It is obvious that $DC=\kappa\,C^*C$  and $C^*C+DD^*=(1+|\kappa|^2)\,C^*C$ are under  the  hypothesis of Theorem \ref{principal} and have, respectively, the eigenvalues 
$$z_j=\kappa s^2_j  \qquad\hbox{and}\qquad h_j=(1+|\kappa|^2)\,s^2_j, \qquad  j=1, \dots,k,$$ all the remaining being zero if $k<\floor[big]{\frac{n}{2}}$.
The result easily follows from Theorem~\ref{principal}, observing that $z_1\geq \dots \geq z_k$ and $h_1 \geq \cdots \geq h_k$ allow us to conclude that the hyperbolas ${\cal H}_1, \dots, {\cal H}_k$ are nested and non-degenerated if and only if (\ref{condTrid1}) holds; the outer one ${\cal H}_1$ forms the hyperbolic boundary of $W^{\tilde J}(T)$.
\end{proof}

\medskip


\section{Final  comments}

\smallskip

We have investigated  classes of $2$-by-$2$ block matrices $A$, 
 with indefinite boundary generating curve of $W^J(A)$ consisting of concentric hyperbolas, and eventually points, which generalize some previous known results. In our study, the diagonal blocks of these matrices $A$ are scalar multiples of the identity.
 Using a unified approach, cases when arcs of hyperbolas and flat portions appear on the boundary of the Krein space numerical range were presented, as well as cases when the set has precisely a hyperbolic boundary.
The main theorem was further applied to reobtain a result concerning some tridiagonal matrices with biperiodic main diagonal and other characterizations for the Krein space numerical range of these type of matrices could be analogously derived.

There is still much to investigate, other  block matrices or structured matrices could be considered, the case of non-concentric hyperbolas could be studied, and corresponding flat portions appearing at the boundary of $W^J(A)$ be characterized, nevertheless  much more involved computations that  lead to less elegant characterizations of the boundary generating curves of $W^J(A)$ are expected.

\subsection*{Acknowledgment}

The  first author was partially supported by the Centre for Mathema\-tics of the University of Coimbra (funded by the Portuguese Government through FCT/MCTES. 
The se\-cond author was supported by CIDMA under the Portuguese Foundation for Science and Technology 
(FCT, https://ror.org/00snfqn58) Multi-Annual Financing Program for R\&D Units.
The third author was financed by  CMAT-UTAD through the Portuguese Foundation for Science and Technology (FCT - Fundac\~ao para a Ci\^encia e a Tecnologia). 

\end{document}